\documentclass[reqno]{amsart}
\usepackage{amsfonts}%
\usepackage{amssymb}
\usepackage{mathrsfs}%
\usepackage{cite}
\usepackage{color}
\allowdisplaybreaks
\date{\today}
\usepackage{mathrsfs}
\usepackage{amssymb}
\usepackage{amsmath}
\usepackage{amsthm}
\usepackage{latexsym}
\usepackage{graphicx}

 \numberwithin{equation}{section}

\newcommand{\dv}{\mathrm{div}\,}
\newcommand{\cl}{\mathrm{curl}\,}
\newcommand{\R}{{\mathbb{R}^3}}
\newcommand{\rt}{\mathrm{rot}\,}

\newtheorem{Theorem}{Theorem}[section]
\newtheorem{Lemma}[Theorem]{Lemma}

\begin{document}

\title[Regularity criterion for full MHD equations]
 { A regularity criterion for the 3D full compressible magnetohydrodynamic equations  with zero heat conductivity}

\author[J.-S. Fan]{Jishan Fan}
\address{ Department of Applied Mathematics,
 Nanjing Forestry University, Nanjing 210037, P.R.China}
\email{fanjishan@njfu.edu.cn}

\author[F.-C. Li]
{Fucai Li}
\address{Department of Mathematics, Nanjing University, Nanjing
 210093, P.R. China}
 \email{fli@nju.edu.cn}
%
 \author[G. Nakamura]{Gen Nakamura}
\address{Department of Mathematics,  Hokkaido University, Sapporo, 060-0810, Japan}
 \email{nakamuragenn@gmail.com}

%

\begin{abstract}
We establish a regularity criterion for the 3D full compressible magnetohydrodynamic equations  with zero heat conductivity and vacuum in a bounded domain.
 \end{abstract}


\maketitle

\section{Introduction}

In this paper, we consider the 3D full compressible magnetohydrodynamic equations in a bounded domain $\Omega\subset \mathbb{R}^3$:
\begin{align}
&\partial_t\rho+\dv(\rho u)=0\label{1.1}\\
&\partial_t(\rho u)+\dv(\rho u\otimes u)+\nabla p-\mu\Delta u-(\lambda+\mu)\nabla\dv u=\rt b\times b,\label{1.2}\\
&C_V[\partial_t(\rho\theta)+\dv(\rho u\theta)]-\kappa\Delta\theta+p\dv u=\frac\mu2|\nabla u+\nabla u^t|^2+\lambda(\dv u)^2+\nu|\rt b|^2,\label{1.3}\\
&\partial_tb+\rt(b\times u)=\nu\Delta b,\ \ \dv b=0,\label{1.4}
\end{align}
with the initial and boundary conditions
\begin{align}
&u=0,\ \ \kappa\frac{\partial\theta}{\partial n}=0,\ \ b\cdot n=0, \ \  \rt b\times n=0\ \ \mathrm{on}\ \ \partial\Omega\times(0,\infty),\label{1.5}\\
&(\rho,u,\theta,b)(\cdot,0)=(\rho_0,u_0,\theta_0,b_0)\ \ \mathrm{in}\ \ \Omega\subset\R.\label{1.6}
\end{align}
Here the unknowns $\rho, u, p, \theta$, and $b$ stand for the density, velocity, pressure, temperature, and magnetic field, respectively. The physical constants $\mu$ and $\lambda$ are the shear viscosity and bulk viscosity of the fluid and satisfy $\mu>0$ and $\lambda+\frac23\mu\geq0$. $C_V>0$ is the specific heat at constant volume and $\kappa>0$ is the heat conductivity. $\nu>0$ is the magnetic diffusivity. $\nabla u^t$ denotes the transpose of the matrix $\nabla u$. We assume that $\Omega$ is a bounded and simply connected domain in $\R$ with smooth boundary $\partial\Omega$. We use $n$ to denote the outward unit normal vector to $\partial\Omega$.

The full  compressible magnetohydrodynamic equations \eqref{1.1}-\eqref{1.4} can be rigorous derivation from the compressible Navier-Stokes-Maxwell system \cite{JL}.
Due to the physical importance of the magnetohydrodynamics, there are a lot of literature on the system \eqref{1.1}-\eqref{1.4}, among others, we mention
\cite{1} on the local strong solutions, \cite{2,3,4} on  the global weak solutions, \cite{JJL3,6} on  low Mach number limit, and
 \cite{7} on the  time decay of smooth small solutions.

%

Assume that the pressure take the form $p=R\rho\theta$ with $R$ being the generic gas constant.

In \cite{8}, Huang and Li proved the following regularity criterion
\begin{equation}
\rho\in L^\infty(0,T;L^\infty)\ \ \mathrm{and}\ \ u\in L^s(0,T;L^r)\ \ \mathrm{with}\ \ \frac2s+\frac3r=1\ \ \mathrm{and}\ \ 3<r\leq\infty,\label{1.7}
\end{equation}
with $b$ satisfying the homogeneous Dirichlet boundary condition $b=0$ on $\partial\Omega\times(0,\infty)$. Later this result was  generalized in \cite{9} to the case of the boundary condition \eqref{1.5}, i.e.,
\begin{equation}
b\cdot n=0, \quad \rt b\times n=0\ \ \mathrm{on}\ \ \partial\Omega\times(0,\infty).\label{1.8}
\end{equation}

When considering the system \eqref{1.1}-\eqref{1.4} in a two dimensional domain,  Lu, Chen and Huang \cite{10} showed the following regularity criterion
\begin{equation}
\dv u\in L^1(0,T;L^\infty)\label{1.9}
\end{equation}
with $b$ satisfying the boundary condition $b=0$ on $\partial\Omega\times(0,\infty)$. Here we remark that same result can be proved for $b$ satisfying the
boundary condition: $b\cdot n=0, \rt b=0$ on $\partial\Omega\times(0,\infty)$.  An related weak result was obtained in \cite{13}.

Very recently,  Huang and Wang \cite{11} establish the following regularity criterion
\begin{equation}
\rho,\theta,b\in L^\infty(0,T;L^\infty)\ \ \mathrm{with}\ \ 2\mu>\lambda.\label{1.10}
\end{equation}
for the system \eqref{1.1}-\eqref{1.4} in the whole space $\mathbb{R}^3$ with $\kappa=\nu=0$.

The aim of this paper is to show that the regularity criterion \eqref{1.10} still hold  for 
 the system \eqref{1.1}-\eqref{1.4} in a bounded domain with the boundary condition \eqref{1.5} when $\kappa=0$ and $\nu=1$. We will prove
\begin{Theorem}
\label{th1.1} Let $\kappa=0$ and $\nu=1$. For $q\in(3,6]$, assume that the initial data $(\rho_0\geq0, u_0, p_0= R\rho_0\theta_0\geq0, b_0)$ satisfy
\begin{equation}\label{1.11}
\left\{\begin{array}{l}
 \rho_0,p_0\in W^{1,q}(\Omega),\ u_0\in H_0^1(\Omega)\cap H^2(\Omega),\ b_0\in H^2\ \ \mathrm{with}\ \ \dv b_0=0\ \ \mathrm{in}\ \ \Omega,\\
b_0\cdot n=0,\, \rt b_0\times n=0\ \ \mathrm{on}\ \ \partial\Omega
\end{array} \right.
\end{equation}
and the compatibility condition
\begin{equation}
-\mu\Delta u_0-(\lambda+\mu)\nabla\dv u_0+\nabla p_0-\rt b_0\times b_0=\sqrt{\rho_0}g,\label{1.12}
\end{equation}
with $g\in L^2(\Omega)$. Let $(\rho, u, p, b)$ be a local strong solution to the problem \eqref{1.1}-\eqref{1.6}. If \eqref{1.10} holds true with $0<T<\infty$, then the solution $(\rho, u, p, b)$ can be extended beyond $T>0$.
\end{Theorem}

We mention that when taking $b=0$ in the system \eqref{1.1}-\eqref{1.4}, it is reduced to the full compressible Navier-Stokes system and  a lot of  regularity criteria
can be found in \cite{16,17,18} and the references cited therein.

The remainder of this paper is devoted to the proof of Theorem \ref{th1.1}. We give some preliminaries in section 2 and present the proof of Theorem \ref{th1.1}
in section 3.
Below we shall use the letter $C$ to denote the positive constant which may change from
line to line.

\section{Preliminaries}

First, we consider the boundary value problem for the Lam\'{e} operator $L$
\begin{equation}\label{2.1}
  \left\{\begin{array}{l}
LU\triangleq \mu\Delta U+(\mu+\lambda)\nabla\dv U=F\ \ \mathrm{in}\ \ \Omega,\\
U(x)=0\ \ \mathrm{on}\ \ \partial\Omega.
\end{array}\right.
\end{equation}
Here $U=(U_1,U_2,U_3), F=(F_1,F_2,F_3)$. It is well known that the system \eqref{2.1} is a strongly elliptic system, thus there exists a unique weak solution $U\in H_0^1(\Omega)$ for $F\in W^{-1,2}(\Omega)$.
\begin{Lemma}
\label{le2.1} Let $q\in(1,\infty)$ and $U$ be a solution of \eqref{2.1}. There exists a constant $C$ depending only on $\lambda, \mu, q$ and $\Omega$ such that the following estimates hold:\\
(1) if $F\in L^q(\Omega)$, then
\begin{equation}
\|U\|_{W^{2,q}(\Omega)}\leq C\|F\|_{L^q(\Omega)};\label{2.3}
\end{equation}
(2) if $F\in W^{-1,q}(\Omega)$ (that is, $F=\dv f$ with $f=(f_{ij})_{3\times3},f_{ij}\in L^q(\Omega))$, then
\begin{equation}
\|U\|_{W^{1,q}(\Omega)}\leq C\|f\|_{L^q(\Omega)};\label{2.4}
\end{equation}
(3) if $F=\dv f$ with $f_{ij}=\partial_kh_{ij}^k$ and $h_{ij}^k\in W_0^{1,q}(\Omega)$ for $i,j,k=1,2,3$, then
\begin{equation}
\|U\|_{L^q(\Omega)}\leq C\|h\|_{L^q(\Omega)}.\label{2.5}
\end{equation}
\end{Lemma}

\begin{proof}
The estimates \eqref{2.3} and \eqref{2.4} are classical for strongly elliptic systems, see for example \cite{19}. The estimate \eqref{2.5} can be proved by a duality argument with the help of \eqref{2.3}.
\end{proof}

We need an endpoint estimate for $L$ in the case $q=\infty$. Let $BMO(\Omega)$ stand for the John-Nirenberg space of bounded mean oscillation whose norm is defined by $$\|f\|_{BMO(\Omega)}:=\|f\|_{L^2(\Omega)}+[f]_{BMO},$$ with
\begin{align*}
[f]_{BMO(\Omega)}:=&\sup\limits_{x\in\Omega,r\in(0,d)}\frac{1}{|\Omega_r(x)|}\int_{\Omega_r(x)}|f(y)-f_{\Omega_r(x)}|dy,\\
f_{\Omega_r(x)}:=&\frac{1}{|\Omega_r(x)|}\int_{\Omega_r(x)}f(y)dy.
\end{align*}
Here $\Omega_r(x):=B_r(x)\cap\Omega, B_r(x)$ is a ball with center $x$ and radius $r,d$ is the diameter of $\Omega$ and $|\Omega_r(x)|$ denotes the Lebesque measure of $\Omega_r(x)$.
\begin{Lemma}[\!\cite{20})]
\label{le2.2} If $F=\dv f$ with $f=(f_{ij})_{3\times3}, f_{ij}\in L^\infty(\Omega)\cap L^2(\Omega)$, then $\nabla U\in BMO(\Omega)$ and there exists a constant $C$ depending only on $\lambda, \mu$ and $\Omega$ such that
\begin{equation}
\|\nabla U\|_{BMO(\Omega)}\leq C(\|f\|_{L^\infty(\Omega)}+\|f\|_{L^2(\Omega)}).\label{2.6}
\end{equation}
\end{Lemma}

Let us conclude this section by recalling a variant of the Brezis-Waigner inequality \cite{21}.
\begin{Lemma}[\!\cite{22}]\label{le2.3}
 Let $\Omega$ be a bounded Lipschitz domain in $\R$ and $f\in W^{1,q}$ with $3<q<\infty$. There exists a constant $C$ depending on $q$ and the Lipschitz property of $\Omega$ such that
\begin{equation}
\|f\|_{L^\infty(\Omega)}\leq C(1+\|f\|_{BMO(\Omega)}\ln(e+\|\nabla f\|_{L^q(\Omega)})).\label{2.7}
\end{equation}
\end{Lemma}
\begin{Lemma}[\!\cite{23}]
\label{le2.4}
 Let $b$ be a solution to the Poisson equation $$-\Delta b=f\ \ in\ \ \Omega$$ with the boundary condition $$b\cdot n=0,\rt b\times n=0\ \ on\ \ \partial\Omega.$$

Then there holds
\begin{equation}
\|\nabla^2b\|_{L^p}\leq C\|f\|_{L^p}+C\|\nabla b\|_{L^2}\ \ with\ \ 1<p<\infty.\label{2.8}
\end{equation}
\end{Lemma}

In the following proofs, we will use the Poincar\'{e} inequality \cite{24}:
\begin{equation}
\|b\|_{L^2}\leq C(\|\dv b\|_{L^2}+\|\rt b\|_{L^2})\label{2.9}
\end{equation}
for any $b\in H^1(\Omega)$ with $b\cdot n=0$ or $b\times n=0$ on $\partial\Omega$.

We will also use the inequality \cite{25}:
\begin{equation}
\|\nabla b\|_{L^2}\leq C(\|\dv b\|_{L^2}+\|\rt b\|_{L^2})\label{2.10}
\end{equation}
for any $b\in H^1(\Omega)$ with $b\cdot n=0$ or $b\times n=0$ on $\partial\Omega$.


\section{Proof of Theorem \ref{th1.1}}

This section is devoted to the proof of Theorem \ref{th1.1}, we only need to show a priori estimates. For simplicity, we will take $\nu=C_V=\mathcal{R}=1$.

Testing \eqref{1.2} by $u$, \eqref{1.4} by $b$, summing up the results and using \eqref{1.1} and \eqref{1.10}, we see that
\begin{align*}
&\frac12\frac{d}{dt}\int(\rho|u|^2+|b|^2)dx+\int(\mu|\nabla u|^2+(\lambda+\mu)(\dv u)^2+|\rt b|^2)dx\\
=&-\int u\nabla p dx=\int p\dv u dx\leq\frac{\lambda+\mu}{2}\int(\dv u)^2dx+C,
\end{align*}
which gives
\begin{equation}
\int(\rho|u|^2+|b|^2)dx+\int_0^T\int(|\nabla u|^2+|\rt b|^2) dx dt\leq C.\label{3.1}
\end{equation}

Integrating \eqref{1.3} over $\Omega\times(0,t)$ and using \eqref{1.10} and \eqref{3.1}, we find that
\begin{equation}
\int\rho\theta dx\leq C.\label{3.2}
\end{equation}

By the same calculations as that in \cite{11}, we get
\begin{equation}
\int\rho|u|^4 dx+\int_0^T\int|\nabla u|^2|u|^2 dx dt\leq C.\label{3.3}
\end{equation}

We define $v\in H_0^1$ satisfying
\begin{equation}
Lv:=\mu\Delta v+(\lambda+\mu)\nabla\dv v=\nabla p,\label{3.4}
\end{equation}
and $w:=u-v$.
Thanks to Lemma \ref{le2.1}, for any $1<r<\infty$, there hold
\begin{equation}
\|\nabla v\|_{L^r(\Omega)}\leq C\|p\|_{L^r(\Omega)},\|\nabla^2v\|_{L^r(\Omega)}\leq C\|\nabla p\|_{L^r(\Omega)}.\label{3.5}
\end{equation}
It is easy to see that $w$ satisfies
\begin{equation}
\mu\Delta w+(\lambda+\mu)\nabla\dv w=\rho\dot u-\rt b\times b,\label{3.6}
\end{equation}
Then it follows from Lemma \ref{le2.1} that
\begin{equation}
\|\nabla^2w\|_{L^2(\Omega)}\leq C\|\rho\dot u\|_{L^2(\Omega)}+C\|\rt b\times b\|_{L^2(\Omega)}.\label{3.7}
\end{equation}

Let $E$ be the specific energy defined by $$E:=\theta+\frac{|u|^2}{2}.$$
Then
\begin{align}
&\partial_t\left(\rho E+\frac{|b|^2}{2}\right)+\dv(\rho E u+p u+|b|^2 u)\nonumber\\
=&\frac12\mu\Delta|u|^2+\mu\dv(u\cdot\nabla u)+\lambda\dv(u\dv u)+\dv((u\cdot b)b)-\dv(\rt b\times b).\label{3.8}
\end{align}

Testing \eqref{1.2} by $u_t$ and using \eqref{1.1} and denoting $\dot f:=f_t+u\cdot\nabla f$, we deduce that
\begin{align}
&\frac12\frac{d}{dt}\int(\mu|\nabla u|^2+(\lambda+\mu)(\dv u)^2) dx+\int\rho|\dot u|^2 dx\nonumber\\
=&\int\rho\dot u\cdot(u\cdot\nabla)u dx+\int\left[\left(p+\frac12|b|^2\right)\mathbb{I}_3-(b\otimes b)\right]:\nabla u_t dx\nonumber\\
\leq&\frac18\int\rho|\dot u|^2 dx+C\int|u|^2|\nabla u|^2 dx+\frac{d}{dt}\int\left[\left(p+\frac12|b|^2\right)\mathbb{I}_3-(b\otimes b)\right]:\nabla u dx\nonumber\\
&-\int p_t\dv u dx-\int\left[\frac12|b|^2\mathbb{I}_3-(b\otimes b)\right]_t:\nabla u dx.\label{3.9}
\end{align}
We remark that
\begin{align}
&-\int p_t\dv v dx=\int v\nabla p_t dx=\int v(\mu\Delta v_t+(\lambda+\mu)\nabla\dv v_t)dx\nonumber\\
=&-\frac12\frac{d}{dt}\int(\mu|\nabla v|^2+(\lambda+\mu)(\dv v)^2)dx.\label{3.10}
\end{align}
And according to \eqref{3.8} and \eqref{1.1},
\begin{align}
&-\int p_t\dv w dx\nonumber\\
=&-\int\left(\rho E+\frac12|b|^2\right)_t\dv w dx+\int\left(\frac12\rho|u|^2\right)_t\dv w dx+\int(b\cdot b_t)\dv w dx\nonumber\\
=&-\int\left(\rho E u+p u+|b|^2u-\frac12\mu\nabla|u|^2-\mu(u\cdot\nabla)u-\lambda u\dv u-(u\cdot b)b+\rt b\times b\right)\nabla\dv w dx\nonumber\\
&-\frac12\int\dv(\rho u)|u|^2\dv w dx+\int\rho u_t u\dv w dx+\int b b_t\dv w dx\nonumber\\
=&-\int\left(2\rho\theta u+|b|^2u-\frac12\mu\nabla|u|^2-\mu(u\cdot\nabla)u-\lambda u\dv u-(u\cdot b)b+\rt b\times b\right)\nabla\dv w dx\nonumber\\
&+\int\rho\dot u u\dv w dx+\int b b_t\dv w dx\nonumber\\
\leq&C\int\rho|u|^2 dx+C\|u\|_{L^6}^2+C\int|u|^2|\nabla u|^2 dx+C\|\nabla b\|_{L^2}^2+\delta_1\|\nabla^2w\|_{L^2}^2\nonumber\\
&+\delta_2\|\sqrt\rho\dot u\|_{L^2}^2+C\|\sqrt\rho u\|_{L^4}^2\|\nabla w\|_{L^4}^2+C\|\nabla w\|_{L^2}^2+\delta_3\|b_t\|_{L^2}^2\nonumber\\
\leq&C+C\|\nabla u\|_{L^2}^2+C\int|u|^2|\nabla u|^2 dx+C\|\nabla b\|_{L^2}^2+C\delta_1\|\sqrt\rho\dot u\|_{L^2}^2\nonumber\\
&+\delta_2\|\sqrt\rho\dot u\|_{L^2}^2+C\|\nabla w\|_{L^2}^\frac12\|\nabla^2w\|_{L^2}^\frac32+C\|\nabla w\|_{L^2}^2+\delta_3\|b_t\|_{L^2}^2\nonumber\\
\leq&C+C\|\nabla u\|_{L^2}^2+C\int|u|^2|\nabla u|^2 dx+C\|\nabla b\|_{L^2}^2\nonumber\\
&+C\delta_1\|\sqrt\rho\dot u\|_{L^2}^2+\delta_2\|\sqrt\rho\dot u\|_{L^2}^2+\delta_3\|b_t\|_{L^2}^2\label{3.11}
\end{align}
for any small $0<\delta_1,\delta_2$ and $\delta_3$.
Here we have used the Gagliardo-Nirenberg inequality $$\|\nabla w\|_{L^4}\leq C\|\nabla w\|_{L^2}^\frac14\|\nabla^2w\|_{L^2}^\frac34+C\|\nabla w\|_{L^2}$$ and $$\|\nabla w\|_{L^2}\leq\|\nabla u\|_{L^2}+\|\nabla v\|_{L^2}\leq C+\|\nabla u\|_{L^2}.$$
Observing that the last term of \eqref{3.9} can be bounded as
\begin{equation}
-\int\left[\frac12|b|^2\mathbb{I}_3-(b\otimes b)\right]_t:\nabla u dx\leq\delta_3\|b_t\|_{L^2}^2+C\|\nabla u\|_{L^2}^2.\label{3.12}
\end{equation}

On the other hand, testing \eqref{1.4} by $b_t-\Delta b$, we get
\begin{align}
&\frac{d}{dt}\int|\rt b|^2 dx+\int(|b_t|^2+|\Delta b|^2) dx\nonumber\\
=&\int|\rt(b\times u)|^2 dx\leq C\int|\nabla u|^2 dx+C\|u\|_{L^6}^2\|\nabla b\|_{L^3}^2\nonumber\\
\leq&C\|\nabla u\|_{L^2}^2+C\|u\|_{L^6}^2\|\nabla b\|_{L^2}^2+\frac12\|\Delta b\|_{L^2}^2+C\|u\|_{L^6}^4\|\nabla b\|_{L^2}^2.\label{3.13}
\end{align}
Here we have used the inequality
\begin{equation}
\|\nabla^2b\|_{L^2}\leq C\|\Delta b\|_{L^2}+C\|\nabla b\|_{L^2}\label{3.14}
\end{equation}
and the Gagliardo-Nirenberg inequality
\begin{equation}
\|\nabla b\|_{L^3}^2\leq C\|\nabla b\|_{L^2}\|\nabla^2b\|_{L^2}+C\|\nabla b\|_{L^2}^2.\label{3.15}
\end{equation}

Inserting \eqref{3.10}, \eqref{3.11} and \eqref{3.12} into \eqref{3.9} and combining \eqref{3.13} and choosing $\delta_1, \delta_2$ and $\delta_3$ suitably small and using the Gronwall inequality, we have
\begin{equation}
\sup\limits_{0\leq t\leq T}\int(|\nabla u|^2+|\nabla b|^2)dx+\int_0^T\int(|\sqrt\rho u_t|^2+|b_t|^2+|\nabla^2b|^2) dx dt\leq C.\label{3.16}
\end{equation}

Now we are in a position to give a high order regularity estimates of the solutions. The calculations were motivated by \cite{16}. First of all, we rewrite the equation \eqref{1.2} as $$\rho\dot u+\nabla p-Lu=g:=\rt b\times b$$ to find that
\begin{align*}
&\rho\dot u_t+\rho u\cdot\nabla\dot u+\nabla p_t+\dv(\nabla p\otimes u)\\
=&\mu[\Delta u_t+\dv(\Delta u\otimes u)]+(\lambda+\mu)[\nabla\dv u_t+\dv(\nabla\dv u\otimes u)]+g_t+\dv(g\otimes u).
\end{align*}
Testing the above equation by $\dot u$ and using \eqref{1.1}, we have
\begin{align}
&\frac12\frac{d}{dt}\int\rho|\dot u|^2 dx-\mu\int\dot u[\Delta u_t+\dv(\Delta u\otimes u)] dx\nonumber\\
&-(\lambda+\mu)\int\dot u[\nabla\dv u_t+\dv(\nabla\dv u\otimes u)] dx\nonumber\\
=&\int(p_t\dv\dot u+(u\cdot\nabla)\dot u\cdot\nabla p) dx+\int(g_t+\dv(g\otimes u))\dot u dx.\label{3.17}
\end{align}
As in \cite{16}, one can estimate the second and third terms in above equation as follows. $$-\int\dot u[\Delta u_t+\dv(\Delta u\otimes u)] dx\geq\int\left(\frac34|\nabla\dot u|^2-C|\nabla u|^4\right) dx,$$ and $$-\int\dot u[\nabla\dv u_t+\dv(\nabla\dv u\otimes u)] dx\geq\int\left(\frac12(\dv\dot u)^2-\frac18|\nabla\dot u|^2-C|\nabla u|^4\right) dx.$$
Since $p:=\rho\theta$, we rewrite \eqref{1.3} as follows,
\begin{equation}
p_t+\dv(p u)+p\dv u=\frac\mu2|\nabla u+\nabla u^t|^2+\lambda(\dv u)^2+|\rt b|^2.\label{3.18}
\end{equation}
Using \eqref{3.18}, as in \cite{16,11}, one can estimate the fourth term in \eqref{3.17} as follows.
\begin{equation}
\int(p_t\dv\dot u+(u\cdot\nabla)\dot u\cdot\nabla p) dx\leq C+C\|\nabla u\|_{L^4}^4+C\|\nabla b\|_{L^4}^4+\frac\mu8\|\nabla\dot u\|_{L^2}^2.\label{3.19}
\end{equation}
Using $b\cdot\nabla b+b\times\rt b=\frac12\nabla|b|^2$, and \eqref{3.16}, we bound the last term of \eqref{3.17} as follows.
\begin{align*}
&\int(g_t+\dv(g\otimes u))\dot u dx\\
=&\int\left[\dv\left(\frac12|b|^2\mathbb{I}_3-b\otimes b\right)+\dv(g\otimes u)\right]\dot u dx\\
=&-\int\left(\frac12|b|^2\mathbb{I}_3-b\otimes b+g\otimes u\right):\nabla\dot u dx\\
\leq&C(\|b\|_{L^4}^2+\|b\|_{L^\infty}\|\rt b\|_{L^3}\|u\|_{L^6})\|\nabla\dot u\|_{L^2}\\
\leq&C(1+\|\nabla b\|_{L^3})\|\nabla\dot u\|_{L^2}\leq\frac\mu8\|\nabla\dot u\|_{L^2}^2+C\|\nabla b\|_{L^3}^2+C.
\end{align*}
Inserting the those estimates into \eqref{3.16} and using
\begin{align*}
&\|\nabla b\|_{L^4}^4\leq C\|b\|_{L^\infty}^2\|b\|_{H^2}^2,\\
&\|\nabla u\|_{L^4}^4\leq\|\nabla u\|_{L^2}\|\nabla u\|_{L^6}^3\leq C(\|\nabla v\|_{L^6}+\|\nabla w\|_{L^6}^3)\leq C(1+\|\sqrt\rho\dot u\|_{L^2}^3),
\end{align*}
We have
\begin{align*}
&\frac12\frac{d}{dt}\int\rho|\dot u|^2 dx+\mu\int|\nabla\dot u|^2 dx+(\lambda+\mu)\int(\dv\dot u)^2 dx\\
\leq&C(1+\|\nabla u\|_{L^4}^4+\|b\|_{H^2}^2)\\
\leq&C+C\|\sqrt\rho\dot u\|_{L^2}^4+C\|b\|_{H^2}^2,
\end{align*}
which gives
\begin{equation}
\|\sqrt\rho\dot u\|_{L^\infty(0,T;L^2)}+\|\dot u\|_{L^2(0,T;H^1)}\leq C.\label{3.20}
\end{equation}

By the same calculations as in \cite{11}, it is easy to verify that
\begin{equation}
\sup\limits_{0\leq t\leq T}\|\nabla w\|_{H^1}+\int_0^T(\|\nabla^2w\|_{L^p}^2+\|\nabla w\|_{L^\infty}^2) dt\leq C\ \ \mathrm{with\ \ any}\ \ 2\leq p\leq6.\label{3.21}
\end{equation}

Applying $\partial_t$ to \eqref{1.4}, testing the result by $b_t$, using \eqref{3.3} and \eqref{3.20}, we have
\begin{align*}
&\frac12\frac{d}{dt}\int|b_t|^2 dx+\int|\rt b_t|^2 dx=-\int\partial_t(b\times u)\rt b _t dx\\
=&-\int(b_t\times u+b\times\dot u-b\times(u\cdot\nabla)u)\rt b_t dx\\
\leq&(\|b_t\|_{L^3}\|u\|_{L^6}+\|b\|_{L^3}\|\dot u\|_{L^6}+\|b\|_{L^\infty}\|u\cdot\nabla u\|_{L^2})\|\rt b_t\|_{L^2}\\
\leq&C(\|b_t\|_{L^3}+\|\dot u\|_{L^6}+\|u\cdot\nabla u\|_{L^2})\|\rt b_t\|_{L^2}\\
\leq&\frac12\|\rt b_t\|_{L^2}^2+C\|b_t\|_{L^2}^2+C\|\dot u\|_{L^6}^2+C\|u\cdot\nabla u\|_{L^2}^2,
\end{align*}
which implies
\begin{equation}
\|b_t\|_{L^\infty(0,T;L^2)}+\|b_t\|_{L^2(0,T;H^1)}\leq C.\label{3.22}
\end{equation}

This and \eqref{1.4} and \eqref{3.16} lead to
\begin{equation}
\|b\|_{L^\infty(0,T;H^2)}+\|b\|_{L^2(0,T;W^{2,6})}\leq C,\label{3.23}
\end{equation}
where we used $$\|u\|_{L^\infty(0,T;W^{1,6})}\leq\|v\|_{L^\infty(0,T;W^{1,6})}+\|w\|_{L^\infty(0,T;W^{1,6})}\leq C.$$

Direct calculations show that
\begin{equation}
\frac{d}{dt}\|\nabla\rho\|_{L^q}\leq C(1+\|\nabla u\|_{L^\infty})\|\nabla\rho\|_{L^q}+C\|\nabla^2u\|_{L^q},\label{3.24}
\end{equation}
and
\begin{equation}
\frac{d}{dt}\|\nabla p\|_{L^q}\leq C(1+\|\nabla u\|_{L^\infty})(\|\nabla p\|_{L^q}+\|\nabla^2u\|_{L^q})+C\|\nabla b\|_{L^\infty}\|\nabla^2b\|_{L^q}.\label{3.25}
\end{equation}
We bound the last term of \eqref{3.25} as follows.
\begin{equation}
\|\nabla b\|_{L^\infty}\|\nabla^2b\|_{L^q}\leq C(1+\|\nabla^2b\|_{L^q})\|\nabla^2 b\|_{L^q}\leq C+C\|\nabla^2b\|_{L^q}^2.\label{3.26}
\end{equation}

As in \cite{11}, it is easy to prove that
\begin{align}
&\|\nabla\rho\|_{L^\infty(0,T;L^q)}+\|\nabla p\|_{L^\infty(0,T;L^q)}\leq C,\label{3.27}\\
&\|\nabla u\|_{L^2(0,T;L^\infty)}+\|u\|_{L^\infty(0,T;H^2)}\leq C.\label{3.28}
\end{align}

This completes the proof.
\hfill$\square$

\medskip\medskip
{\bf Acknowledgements:}
 Fan is supported by NSFC (Grant No. 11171154). Li is supported partially by NSFC (Grant No. 11271184) and
   PAPD.

\end{document}